\documentclass[12pt]{amsart}

\usepackage{my_preamble}
\usepackage{amssymb, amsmath, amsfonts}
\usepackage[initials]{amsrefs}

\newcommand{\opr}{\mathrm{opR}}
\newcommand{\dpr}{\mathrm{dpR}}
\newcommand{\opd}{\mathrm{opD}}
\newcommand{\UU}{\mathbb{U}}
\newcommand{\mlo}{\mathrm{MLO}}
\newcommand{\OG}{\mathrm{OG}}
\newcommand{\EM}{\textsf{EM}}
\renewcommand{\qtp}{\mathrm{qftp}}
\renewcommand{\C}{\mathcal{C}}

\begin{document}

\title[op-Dimension]{On a Common Generalization of Shelah's 2-Rank, dp-Rank, and o-Minimal Dimension.}
\author{Vincent Guingona,\\ Cameron Donnay Hill}
\address[Guingona]{University of Notre Dame \\ Department of Mathematics \\ 255 Hurley, Notre Dame, IN 46556}
\email{guingona.1@nd.edu}
\urladdr{http://www.nd.edu/~vguingon/}
\address[Hill]{Wesleyan University \\ Department of Mathematics and Computer Science \\  45 Wyllys Avenue, Middletown, CT 06459}
\email{cdhill@wesleyan.edu}
\date{\today}

\begin{abstract}
 In this paper, we build a dimension theory related to Shelah's $2$-rank, dp-rank, and o-minimal dimension.  We call this dimension op-dimension.  We exhibit the notion of the $n$-multi-order property, generalizing the order property, and use this to create op-rank, which generalizes $2$-rank.  From this we build op-dimension.  We show that op-dimension bounds dp-rank, that op-dimension is sub-additive, and op-dimension generalizes o-minimal dimension in o-minimal theories.
\end{abstract}

\maketitle

\section*{Introduction}

At the beginning of this century, the study of dependent/NIP theories experienced something of a renaissance after a number of years of dormancy. With the exception of o-minimal theories (which are, of course, dependent, but this fact saw little actual use), most model theorists' attention had been directed towards stable and then simple theories. However, many ``natural'' algebraic examples turn out to be unstable, non-o-minimal but dependent (sometimes with stronger conditions than bare NIP), such as $p$-adic fields, definably compact groups arising in o-minimal structures, and ordered abelian groups. S. Shelah initiated a careful study of dependent/NIP theories in the series of papers, \cites{Shelah715, Shelah783, Shelah863, Shelah900}; in this work, he defined certain sub-classes of dependent theories known (aptly) as strongly-dependent theories. A key tool in the development of strong-dependence is notion of dp-rank, which to some degree, resembles the notion of weight in a stable theory. 

dp-Rank, $\dpr(-)$, though it really is not a dimension, has some aspects that make it dimension-like. In particular, dp-rank is sub-additive in the sense that $\dpr(ab/C)\leq \dpr(a/C)+\dpr(b/C)$, but without some significant effort to understand forking-dependence for a type of finite dp-rank, it can be relatively difficult initially to see $\dpr$ as a geometric construct. It is also somewhat difficult to accommodate $\dpr$ in the universe of pre-existing model-theoretic definitions. For example, a stable theory need not be strongly-dependent (have finite dp-rank for all types). Thus, while strong-dependence and dp-rank inherit ideas and intuitions from stability theory, they do not formally generalize it. Finally, dp-rank does not (to our knowledge) fit into the evolving framework of generalized-indiscernible ``collapse'' characterizations of model-theoretic dividing lines. That is to say, stability is equivalent to collapsing indiscernible sequences (linear orders) to indiscernible sets; dependence is equivalent to collapsing ordered-graph indiscernibles to indiscernible sequences; but there is no obvious analog even for theories of bounded dp-rank.

In this article, we define an analog of dp-rank -- op-dimension, $\opd$ -- that seems to remedy some of these ``deficiencies.'' Building atop a family of local op-ranks, we find that op-dimension has a number of intuitively desirable properties, including the following:
\begin{itemize}
\item The local ranks $\opr_n$ ($0<n<\omega$) naturally generalize Shelah's 2-rank $R(-,-,2)$ to ``multi-orders'' and ``multi-cuts''; in fact, these $\opr_n$'s formally generalize the classical 2-rank in that $R(-,-,2)=\opr_1(-,-)$.

\item op-Dimension has a loosely, but still explicitly topological flavor. Indeed, in an o-minimal theory, the op-dimension of a definable set is identical to its o-minimal dimension (which also equals its dp-rank), and any theory that ``sub-interpretable'' in an o-minimal theory (in certain weak sense) must be of finite op-dimension. 

Moreover, $\opd$ retains the dimension-like aspects of dp-rank over all strongly-dependent theories; that is, $\opd$ has the appropriate monotonicity properties, and it is sub-additive. 

\item The condition of bounded op-dimension for a theory can be (fruitfully, it seems) understood as a generalization of stability to multi-orders. For each $n$, we will find ourselves with an $n$-multi-order property ($n$-MOP), and 1-MOP is precisely the classical order property. 
\end{itemize}
We will also see that the op-dimension of a type can be characterized in a manner very similar to the definition of dp-rank -- simply replacing ICT-patterns with the closely related IRD-patterns; from this observation, we will show that op-dimension is always bounded by dp-rank. Thus, it appears that op-dimension has a part to play in any strongly-dependent theory.

It should be noted that op-dimension closely resembles what Shelah calls ``$\kappa_{\text{ird}}(T)$'' (see Definition III.7.1 of \cite{Shelah}).  This is also discussed in Section 5 of \cite{Adler}.  We will touch on this fact more when we discuss IRD-patterns.

It is beyond the scope of this paper, but we must also remark, that the condition of bounded op-dimension \emph{can} be characterized by ``collapse'' of certain generalized indiscernibles (in the sense of \cite{Hill}); a little more precisely, op-dimension $n$ is equivalent to asserting that every indiscernible $(n+1)$-multi-order collapses to an indiscernible $n$-multi-order. (As this result is an example of a rather more general phenomenon, we will save it for a more extended discussion of the latter; see \cite{GuinHillScow}.)

\subsection{Outline of the Article}

In Section 1, we outline the basic definitions and results surrounding op-rank and op-dimension.  We introduce various ways of viewing op-dimension, first as a generalization of $2$-rank, then through the lens of the multi-order property, then through its relationship to dp-rank and IRD-patterns.  In Section 2, we give two proofs of the sub-additivity of op-dimension.  The first proof follows the path of \cite{KOUdpmin} using a modified notion of mutually indiscernible sequences.  The second proof uses the multi-order property and has the flavor of a stability argument (e.g., Lascar's inequality).  Finally, in Section 3, we look at op-dimension in the special case of o-minimal theories.  We show that op-dimension, dp-rank, and o-minimal dimension coincide in this case and we discuss interpretations of structures with finite op-dimension in o-minimal structures.

\subsection{Notation}

In this paper, we will work in a language $\L$, a complete $\L$-theory $T$, and a monster model $\UU$.  We will denote tuples of variables by $x$ (instead of $\overline{x}$).  By $\UU^x$ we mean all elements of $\UU$ of the sort of $x$.  For a formula $\phi(x)$, we let $\phi(x)^1 = \phi(x)$ and $\phi(x)^0 = \neg \phi(x)$.  For a set $A \subseteq \UU$, let $e\diag(A)$ denote the elementary diagram of $A$, which is a set of $\L(A)$ formulas.  Let $\diag(A)$ denote the atomic diagram of $A$.  Let $1_S$ denote the identity permutation on a set $S$.

Given a formula $\phi(x,y)$ and a set $B \subseteq \UU^y$, let $S_\phi(B)$ denote the set of all $\phi$-types over $B$, by which we mean maximally consistent subsets of the set
\[
 \{ \phi(x,b)^t : t < 2, b \in B \}.
\]
The \emph{independence dimension} of $\phi$ is the size of the largest finite set $B \subseteq \UU^y$ so that
\[
 |S_\varphi(B)| = 2^{|B|}.
\]
If no such largest set exists, we say that $\phi$ has the \emph{independence property} (IP).  If it does exist, we say $\phi$ has \emph{NIP} (sometimes called ``dependent'').  A theory has \emph{NIP} if all formulas have NIP.

\bigskip

\section{Definitions and Basic Results}

Either ``under the hood'' or explicitly, the notions of multi-order and multi-cut together play an important role in much of the work in this article. In part to motivate these definitions, we begin our discussion with a somewhat eccentric definition of classical order property (which by compactness, is equivalent to the usual statement). In practice, we will work with ``$n$-multi-orders,'' but in particular, a linear order is a 1-multi-order. In a linear order $(B,<)$, of course, a {\em cut} is a subset $X\subseteq B$ such that for all $b_0,b_1\in B$, if $b_0<b_1$ and $b_1\in X$, then $b_0\in X$. 

\begin{defn}
Let $\phi(x,y)$ be some formula of $\L(\UU)$. We say that $\phi(x,y)$ has the order property if there is an indiscernible sequence $(a_q)_{q\in\QQ}$ (of sort $x$) such that for every cut $Z$ in $\QQ$, there is a $b_Z\in \UU^y$ such that $Z = \left\{q\in \QQ: \,\,\models\phi(a_q,b_Z)\right\}$.
\end{defn}

Somewhat strangely, an $n$-multi-order is not (in general) just the cartesian product of $n$ linear orders. (Otherwise, we would not have invented the terminology, obviously.) Instead, an $n$-multi-order is a set $B$ equipped with $n$ linear orderings that do not {\em a priori} have any dependencies. Further, rather than working with arbitrary multi-orders, it is much more convenient to observe, firstly, that the common universal theory of all $n$-multi-orders, $\mlo_n$ (for some fixed $0<n<\omega$) has a model-companion $\mlo_n^*$ that has enough in common with the theory of $(\QQ,<)$ to be useful  to us (in fact, $\mlo_1^* =$ DLO).

\begin{defn}
For each $0<n<\omega$, we define two closely related theories $\mlo_n$ and $\mlo_n^*$ with signature $\{<_0,...,<_{n-1}\}$, where each $<_i$ is a binary relation symbol. $\mlo_n$ asserts that each $<_i$ is a linear order of the universe \emph{and nothing else}.

It is not difficult to verify that the class $K_n$ of all finite models of $\mlo_n$ is a \fraisse class. We take $\A_n = \left(A,<_0^{\A_n},...,<_{n-1}^{\A_n}\right)$ to be the countably infinite generic model (or \fraisse limit) associated with $K_n$, and we define $\mlo_n^* = Th(\A_n)$. Then $\mlo_n^*$ is just the model-companion of $\mlo_n$, and by old results (see \cite{Hodges}), $\mlo_n^*$ is $\aleph_0$-categorical and eliminates quantifiers.
\end{defn}

\begin{fact}
Consider the $\{<_i\}_{i<n}$ on $\QQ^n$ in which,
$$<_i^{\QQ^n} = \left\{(\aa,\bb)\in\QQ^n\times\QQ^n:a_i<b_i\right\}$$
for each $i<n$. Then, $\QQ^n$ \emph{is not} a model of $\mlo_n^*$. To see this, one may note that (for example),
$$\mlo_2^*\models\forall xy\left(x\neq y\cond\bigwedge_{i<2}(x<_iy\vee y<_ix)\right)$$
but in $\QQ^2$, $(0,1)\neq (0,2)$ but $(0,1)\not<_0(0,2)$ and $(0,2)\not<_0(0,1)$.
\end{fact}

\begin{defn}
For $0<n<\omega$, any model of $\mlo_n$ is called an {\em $n$-multi-order}. 

Now, if $\B = (B,<_0,...,<_{n-1})$ is a model of $\mlo_n$, then a multi-cut (an $n$-multi-cut) in $\B$ is a tuple $(X_0,...,X_{n-1})$ such that $X_i$ is a cut in the reduct $(B,<_i)$ for each $i<n$.
\end{defn}

To conclude these introductory remarks, we note that the potential to define all cuts in an indiscernible copy of $(\QQ,<)$ is captured by Shelah's 2-rank, and insofar as $\mlo_n^*$ is similar enough to DLO, much of the insight of this article lies in the observation that analogous ranks, $\opr_n$, can be devised to capture the potential of defining all multi-cuts in a model of $\mlo_n^*$.

\subsection{op-Ranks and the op-Dimension of a Type}

In this subsection, we introduce our analogs of Shelah's 2-rank -- of which there will one rank for each $0<n<\omega$ corresponding to the number of independent linear orders in an $n$-multi-order. Several of the most basic facts about $\opr_n$s are themselves totally analogous to those regarding the 2-rank with almost identical proofs. In our presentation, to begin with anyway, we recall the definitions associated with the 2-rank and remind the reader of the relevant facts, and then we give analogous definitions for $\opr_n$ and the corresponding facts (without proof as those demonstrations are almost identical).

\begin{defn}
For a (consistent) partial type $\pi(x)$ and a finite set $\Delta$ of partitioned formulas $\phi(x,y)\in\L(\UU)$, we recall that the Shelah 2-rank of $\pi(x)$ with respect to $\Delta$ is defined as follows:
\begin{itemize}
\item $R(\pi,\Delta,2)\geq0$ in any case.
\item For a limit ordinal $\lambda$, $R(\pi,\Delta,2)\geq\lambda$ if $R(\pi,\Delta,2)\geq\alpha$ for every $\alpha<\lambda$.
\item For any ordinal $\alpha$, $R(\pi,\Delta,2)\geq\alpha+1$ if there is an instance $\phi(x,a)$ from $\Delta$ such that $R(\pi\cup\{\phi(x,a)^t\},\Delta,2)\geq\alpha$ for both $t<2$.
\end{itemize}
As usual, we define $R(\pi,\Delta,2)=\infty$ to mean that $R(\pi,\Delta,2)\geq\alpha$ for every ordinal $\alpha$. When $\Delta = \{\phi\}$ consists of a single formula, one usually writes $R(-,\phi,2)$ in place of $R(-,\{\phi\},2)$.

For an ordinal $\beta$, $\Gamma_\lambda(\pi,\phi)$ is the following set of sentences (with new constant symbols $a_\sigma$, $b_{\sigma\r\ell}$ for $\sigma\in 2^\beta$ and $\ell<\beta$):
$$\bigcup_{\sigma\in 2^\beta}\pi(a_\sigma)\,\cup\, \left\{\phi(a_\sigma,b_{\sigma\r \ell})^{\sigma(\ell)}:\sigma\in 2^\beta,\,\ell<\beta\right\}.$$
\end{defn}

The first basic result about the 2-rank is the following (coming from straightforward applications of compactness and ``coding tricks''). 

\begin{fact}
Let $\pi(x)$ be a partial type, and let $\Delta$ be a finite set of formulas of $\L(\UU)$. Also, let $\phi(x,y)\in\L(\UU)$.
\begin{enumerate}
\item By compactness, $R(\pi,\Delta,2)=\infty$ if and only if $R(\pi,\Delta,2)\geq\omega$.
\item For any ordinal $\beta$, $R(\pi,\phi,2)\geq\beta$ if and only if $\Gamma_\beta(\pi,\phi)\cup e\diag(\UU)$ is consistent. 
\end{enumerate}
Also, for any finite set $\Delta$ of formulas $\theta(x,y)$ of $\L(\UU)$, there is a single formula $\phi_\Delta(x,z)$ of $\L(\UU)$ such that $R(-,\Delta,2) = R(-,\phi_\Delta,2)$.
\end{fact}

Now, we turn to our family of analogs of the 2-rank. For each parameter $0<n<\omega$, the ``key'' distinction between the 2-rank and $\opr_n$ lies in replacing the trees $2^{<\omega}$ -- whose nodes are maps $\sigma:k\to 2$ ($k<\omega$) -- with trees $(2^n)^{<\omega}$ whose nodes are of the form $\sigma:k\to 2^n$ ($k<\omega$); an element of $2^n$, here, represents a particular multi-cut in a model of $\mlo_n$.

\begin{defn}
For $0<n<\omega$, a (consistent) partial type $\pi(x)$ and a finite set $\Delta$ of partitioned formulas $\phi(x,y)\in\L(\UU)$, we define $\opr_n(\pi,\Delta)$ as follows:
\begin{itemize}
\item $\opr_n(\pi,\Delta)\geq0$ in any case.
\item For a limit ordinal $\lambda$, $\opr_n(\pi,\Delta)\geq\lambda$ if $\opr_n(\pi,\Delta)\geq\alpha$ for every $\alpha<\lambda$.
\item For any ordinal $\alpha$, $\opr_n(\pi,\Delta)\geq \alpha+1$ if there are instances 
$$\phi_0(x,a_0),...,\phi_{n-1}(x,a_{n-1}) \text{ from } \Delta$$
such that for each $\sigma\in 2^n$, 
$$\opr_n\left(\pi(x)\cup\left\{\bigwedge_{i<n}\phi_i(x,a_i)^{\sigma(i)}\right\},\Delta\right)\geq\alpha$$
\end{itemize}
Again, we define $\opr_n(\pi,\Delta)=\infty$ to mean that $\opr_n(\pi,\Delta)\geq\alpha$ for every ordinal $\alpha$. When $\Delta = \{\phi\}$ consists of a single formula, we write $\opr_n(-,\phi)$ in place of $\opr_n(-,\{\phi\})$.


For an ordinal $\beta$, $\Gamma_{n,\beta}(\pi,\phi)$ is the following set of sentences (with new constant symbols $a_\sigma$, $b_{\sigma,\ell,0},...,b_{\sigma,\ell,n-1}$ for $\sigma = (\sigma_\ell)_{\ell<\beta}\in (2^n)^\beta$ and $\ell<\beta$):
$$\bigcup_{\sigma\in (2^n)^\beta}\pi(a_\sigma)\,\cup\, \left\{\phi(a_\sigma,b_{\sigma, \ell,i})^{\sigma_\ell(i)}:\sigma\in (2^n)^\beta,\,\ell<\beta\right\}.$$
{\em Implicitly, we require that for all $\sigma,\tau\in (2^n)^\omega$, $\ell<\omega$, if $\sigma_k=\tau_k$ for each $k<\ell$, then $b_{\sigma,\ell,i} = b_{\tau,\ell,i}$ for each $i<n$.}
\end{defn}

\begin{fact}
Let $\pi(x)$ be a partial type, and let $\Delta$ be a finite set of formulas of $\L(\UU)$. Also, let $\phi(x,y)\in\L(\UU)$ and $0<n<\omega$.
\begin{enumerate}
\item $\opr_n(\pi,\Delta)=\infty$ if and only if $\opr_n(\pi,\Delta)\geq\omega$.
\item For any ordinal $\beta$, $\opr_n(\pi,\phi)\geq\beta$ if and only if $\Gamma_\beta(\pi,\phi)\cup e\diag(\UU)$ is consistent. 
\end{enumerate}
Also, for any finite set $\Delta$ of formulas $\theta(x,y)$ of $\L(\UU)$, there is a single formula $\phi_\Delta(x,z)$ of $\L(\UU)$ such that $\opr_n(-,\Delta) = \opr_n(-,\phi_\Delta)$.
\end{fact}


The closest analog to op-dimension in the stability theory literature is the notion of $\kappa_{\text{ird}}(T)$ defined in \cite{Shelah}.  However, this concept is approached through the notion of an IRD-pattern and not through a $2$-rank-like construction.  In the unstable setting, using op-ranks, we can define op-dimension in a very simpleminded way.

\begin{defn}
For a partial type $\pi(x)$, we define the {\em op-dimension of $\pi(x)$} to be, 
$$\opd(\pi) = \sup\left\{0<n<\omega:(\exists \Delta)\,\opr_n(\pi,\Delta) = \infty\right\}\leq\omega.$$
(Note that, by definition of $\sup$ on ordinals, $\sup\emptyset = 0$.) As is standard, for $a\in\UU$ and $B\subset\UU$, we define $\opd(a/B)$ to be $\opd(\tp(a/B))$. For a formula $\phi(x)\in\L(\UU)$, we define $\opd(\phi) = \opd(\{\phi\})$, and if $X$ is the subset of $\UU^x$ defined by $\phi(x)$, then $\opd(X) =\opd(\phi)$.
\end{defn}

\begin{rem}
Let us say that a partial type $\pi(x)$ is \emph{unstable} if there are a formula $\phi(x,y)$ of $\L(\UU)$ and an indiscernible sequence $(a_q)_{q\in\QQ}$ of realizations of $\pi$ such that for every cut $X$ of $(\QQ,<)$, there is a $b\in \UU^y$ such that $\left\{q\in\QQ:\UU\models\phi(a_q,b)\right\} = X$. Obviously, we should say that $\pi(x)$ is stable just in case it is not unstable. Thus, $\pi(x)$ is stable if and only if $\opd(\pi) = 0$.
\end{rem}

The following statement collects together a number of facts whose analogs for the 2-rank are essential in developing the machinery of forking-dependence in a stable theory -- when one carries out that development using ranks, as turned out to be very useful for generalizations to simple and rosy theories. For our purposes, they immediate suggest that $\opd$ can indeed be viewed as a dimension function insofar as it has, at least, the appropriate monotonicity properties of a reasonable dimension theory. 

\begin{fact}
$\opr_n$ ($0<n<\omega$) has the following monotonicity properties:
\begin{enumerate}
\item Suppose $\pi_0(x)\subseteq\pi_1(x)$, $\Delta_0\supseteq\Delta_1$, and $0<n_0\leq n_1<\omega$. Then, 
$$\opr_{n_0}(\pi_0,\Delta_0)\geq\opr_{n_1}(\pi_1,\Delta_1).$$
\item Let $X_0,X_1$ be definable sets of the same sort, $0<n<\omega$, and $\Delta$ a finite set of formulas of $\L(\UU)$. Then 
$$\opr_n(X_0\vee X_1,\Delta) = \max\left\{\opr_n(X_0,\Delta),\opr_n(X_1,\Delta)\right\}.$$
\item Let $X,Y$ be a type-definable sets, and suppose $f:X\to Y$ is definable bijection. Then, for any $0<n<\omega$, for any finite set $\Delta$ of $\L(\UU)$-formulas, there is another finite set of formulas $\Delta'$ such that $\opr_n(X,\Delta) = \opr_n(Y,\Delta')$.
\end{enumerate}
\end{fact}

\begin{cor}\label{cor:dim-monotonicity}
$\opd$ has the following monotonicity properties of a dimension (for type-definable sets $X,Y$):
\begin{enumerate}
\item If $X,Y$ are in definable bijection with each other, then $\opd(X)=\opd(Y)$.
\item If $X\subseteq Y$, then $\opd(X)\leq \opd(Y)$.
\item Provided the definable sets $X,Y$ are of the same sort,
$$\opd(X\vee Y)=\max\left\{\opd(X),\opd(Y)\right\}.$$
\end{enumerate}
\end{cor}

\subsection{Generalized Indiscernibles and $n$-MOP}

In this subsection, we demonstrate some connections between op-dimension and an evolving framework connecting generalized-indiscernible ``collapse'' theorems and dividing lines in the model-theoretic (in)stability hierarchy.

\begin{thm}\label{thm:mlo-theory-of-indisc}
For every $0<n<\omega$, $\mlo_n$ is a theory of generalized indiscernibles in the sense of \cites{GuinHillScow,Hill,KPT}:

Let $\A\models \mlo_n^*$, and let $\M$ be some $|A|^+$-saturated $\L$-structure (in any language $\L$ whatever). Let $\EM$ be a map $A^{<\omega}\to M^{<\omega}$ (really, a family of maps $A^k\to M^k$ for $0<k<\omega$) such that:
\begin{itemize}
\item If $\EM(a_0,...,a_{k-1}) = (b_0,...,b_{k-1})$, then for each $\sigma\in\emph{Sym}(k)$, 
$$\EM(a_{\sigma(0)},...,a_{\sigma(k-1)}) = (b_{\sigma(0)},...,b_{\sigma(k-1)}).$$
\item For any $\aa,\aa'\in A^{<\omega}$, $\tp^\M(\EM(\aa\widehat{\,\,}\aa')) = \tp^\M(\EM(\aa)\widehat{\,\,}\EM(\aa')).$ 
\end{itemize}
Then, there is a map $g:A\to M$ such that:
\begin{itemize}
\item For all $0<k<\omega$ and $\aa,\aa'\in A^k$, 
$$\qtp^\A(\aa) = \qtp^\A(\aa')\,\,\implies\,\,\tp^\M(g\aa) = \tp^\M(g\aa').$$
\item For all $0<k<\omega$, every $\aa\in A^k$, and every finite set $\Delta$ of $\L$-formulas, there is an $\aa'\in A^k$ such that $\qtp^\A(\aa) = \qtp^\A(\aa')$ and $\tp_\Delta^\M(g\aa) = \tp_\Delta^\M(\EM(\aa'))$.
\end{itemize}
(For brevity, we say that $g$ is an indiscernible picture of $\A$ in $\M$ patterned on $\EM$.)
\end{thm}

The proof of Theorem \ref{thm:mlo-theory-of-indisc} can be found in \cite{Hill}. For background on the generalized-indiscernible collapse phenomenon, we cite the following theorem of \cite{Scow}.

\begin{thm}
Let $\OG$ be the theory (in the signature $\left\{<^{(2)},R^{(2)}\right\}$) of ordered graphs; that is, $\OG$ asserts the following:
\begin{itemize}
\item ``$<$ is a linear order of the universe (i.e. of the vertices).''
\item $\forall xy\left[R(x,y)\cond((x\neq y)\wedge R(y,x))\right]$
\end{itemize}
Then, $\OG$ has a model-companion $\OG^*$, which is also the theory of the \fraisse limit of the class of all finite ordered graphs. Moreover:
\begin{enumerate}
\item $\OG$ is a theory of indiscernibles in the same sense (of Theorem \ref{thm:mlo-theory-of-indisc}) that each $\mlo_n$ is.
\item The following are equivalent for any complete theory $T$ in any language whatever:
\begin{enumerate}
\item $T$ is dependent/NIP.
\item For any indiscernible picture $g$ of a model $\A = (A,<^\A,R^\A)$ of $\OG^*$ in an model $\M$ of $T$, $(g(a))_{a\in A}$ is an indiscernible sequence in order type $(A,<^\A)$, in the usual sense.
\end{enumerate}
\end{enumerate}
\end{thm}

Intuitively, this theorem asserts that, for all intents, the theory of indiscernibles $\OG$ encodes the independence property. Viewing the theory of linear order LO ($=\mlo_1$) as a theory of indiscernibles, as we may, the following venerable characterization of stability also fits (loosely) into this framework. (There is actually a mismatch in that the ``remainder'' of $\OG$ in a dependent/NIP theory is $\mlo_1$, which is still a theory of indiscernibles, but the remainder of $\mlo_1$ in a stable theory is the theory of equality, which, in fact, is not a theory of indiscernibles.)

\begin{thm}
Let $T$ be a complete theory in any language. The following are equivalent:
\begin{enumerate}
\item $T$ is unstable.
\item In some model of $T$, there is an indiscernible sequence $(a_q)_{q\in\QQ}$ that is not an indiscernible set.
\end{enumerate}
\end{thm}

We now define the (``smoothed'') combinatorial property that seems to correspond to our op-dimensions in the same way that the order property corresponds to 2-rank.

\begin{defn}
Let $0<n<\omega$, and let $\pi(x)$ be a consistent partial type. We say that $\pi(x)$ has the {\em $n$-multi-order property ($n$-MOP)} if there are an indiscernible picture $(\A_n,g)$ in $\pi(\UU)$ and a formula $\phi(x,y)$ of $\L(\UU)$ such that for any multi-cut $(X_0,...,X_{n-1})$ of $\A_n$, there are $b_0,...,b_{n-1}\in \UU^y$ such that $X_i = \left\{a\in A: \UU\models\phi(g(a),b_i)\right\}$ for each $i<n$.
\end{defn}

We note that the ``collapse'' results in the previous two theorems require a rather fine analysis of exactly how, for example, an ordered-graph indiscernible picture can collapse down to an indiscernible picture of reduct. Such an analysis for our $n$-MOPs would take us outside of the scope of the goals of this paper, though such an analysis will be given in \cite{GuinHillScow}. For now, we consider a more basic analog of the following fact:

\begin{fact}
A partial type $\pi(x)$ is stable iff $R(\pi,\phi,2)<\omega$ for every formula $\phi(x,y)$ of $\L(\UU)$ iff $\opr_1(\pi,\phi)<\omega$ for every formula $\phi(x,y)$ iff $\opd(\pi)=0$.
\end{fact}

\begin{prop}
Let $0<n<\omega$, and let $\pi(x)$ be a consistent partial type. Then, $\opd(\pi)\geq n$ if and only if $\pi(x)$ has $n$-MOP.
\end{prop}
\begin{proof}
(``only if'') Assuming $\opd(\pi(x))\geq n$, let $\phi(x,y)$ be some formula of $\L(\UU)$ such that $\opr_n(\pi(x),\phi)=\infty$. Thus, $e\diag(\UU)\cup\Gamma_{n,\omega}(\pi,\phi)$ is consistent; we recover two families 
$$\{a_\sigma:\sigma\in (2^n)^\omega\},\,\,\left\{b_{\sigma,\ell,i}:\sigma\in (2^n)^\omega,\,\ell<\omega,\,i<n\right\}$$
such that:
\begin{itemize}
\item Each $a_\sigma$ is a realization of $\pi(x)$.
\item For any $\sigma,\tau\in (2^n)^\omega$, $\ell<\omega$, and $i<n$, if $\sigma_j = \tau_j$ for each $j<\ell$, then $b_{\sigma,\ell,i} = b_{\tau,\ell,i}$
\item For any $\sigma\in (2^n)^\omega$, $\ell<\omega$, and $i<n$, $\models\phi(a_\sigma,b_{\sigma,\ell,i})^{\sigma_\ell(i)}$
\end{itemize}
Now, we observe that if $B=(A,<_0,...,<_{n-1})$ is a finite model of $\mlo_n$ with, say, $|B| = N<\omega$, then there is an {\em embedding} $A\to (N^n,<_0,...,<_{n-1})$, where in the latter structure, the orders are interpreted coordinate-wise.
By Theorem \ref{thm:mlo-theory-of-indisc} (and the fact that $\UU$ is $\aleph_1$-saturated), we obtain an injective mapping $g:A\to \UU$ such that $g[A]\subseteq\pi(\UU)$ and for every multi-cut $(X_0,...,X_{n-1})$ of $\A_n$, there are $b_0,...,b_{n-1}\in \UU^y$ such that $X_i = \left\{a\in A:\,\,\models\phi(g(a),b_i)\right\}$ for each $i<n$. Thus, $\pi(x)$ has $n$-MOP.

(``if'') Suppose $\pi(x)$ has $n$-MOP, and let $g:A\to\UU$ and $\phi(x,y)$ witness this fact. We will show that $\opr_n(\pi,\phi)=\infty$, and for this, it is enough to show that for each $N<\omega$, $\Gamma_{n,N}(\pi,\phi)\cup e\diag(\UU)$ is consistent. We observe that for any $N<\omega$, there is an {\em injective homomorphism} of the coordinate-wise ordered structure $((2^N)^n,<_0,...,<_{n-1})$ into $\A_n$, and this suffices for the consistency of $\Gamma_{n,N}(\pi,\phi)\cup e\diag(\UU)$, as required.
 \end{proof}

\subsubsection{A Remark on Localized $\opd$}

We now remark briefly on a localization of op-dimension to finite sets of formulas. It will probably come as no surprise that such a localized rendition of op-dimension amounts to little more than a restatement of the independence property.

\begin{defn}
Let $\pi(x)$ be a partial type. For a finite set $\Delta$ of $\L(\UU)$ formulas, we define,
$$\opd(\pi,\Delta) = \sup\left\{0<n<\omega: \opr_n(\pi,\Delta) = \infty\right\}\leq\omega.$$
\end{defn}

\begin{prop}
The theory $T = Th(\UU)$ has the independence property if and only if $\opd(\{x{=}x\},\Delta) = \omega$ for some tuple $x$ and some finite set $\Delta$ of formulas of $\L(\UU)$.
\end{prop}
\begin{proof}
Assuming $\phi(x,y)$ has the independence property in $T$, we show that $\opd(x{=}x,\phi)=\omega$. We may grant ourselves an indiscernible sequence $(e_a:a\in A)$ (where $A$ is the universe of $\A_n$ equipped with the first order $<_0^{\A_n}$) such that for every $Z\subseteq A$, there is some $b_Z\in \UU^y$ such that $\{a:\,\,\models\phi(e_a,b_Z)\} = Z$.
Given $0<n<\omega$, let $g:A\to \UU$ be an indiscernible picture of $\A_n$ in $\UU$ patterned on 
$$\EM:A^{<\omega}\to \UU^{<\omega}: (a_0,...,a_{k-1})\mapsto (e_{a_0},...,e_{a_{k-1}}).$$
Then, again, for every $Z\subseteq A$, there is a $b_Z\in \UU^y$ such that $\{a:\,\,\models\phi(g(a),b_Z)\} = Z$. Since (of course) multi-cuts are subsets of $A$, this demonstrates that $x{=}x$ has $n$-MOP via $\phi(x,y)$, so $\opd(\{x{=}x\},\Delta) \geq n$

\medskip
Conversely, suppose $\opd(\{x{=}x\},\Delta) = \omega$ for some tuple $x$ and some finite set $\Delta$ of formulas of $\L(\UU)$. Without loss of generality, we may assume that $\opd(\{x{=}x\},\phi) = \omega$ for some single formula $\phi(x,y)$. For a set $B$ of size $N<\omega$, there are $N!$ linear orders on $B$.  Enumerating all of these orders $<_0^B,...,<_{N!-1}^B$, we find ourselves with a finite substructure of $\A_{N!}$. Thus, for any $d<\omega$ one can find arbitrarily large finite sets $B\subset\UU^x$ such that 
$$|S_\phi(B)| = 2^{|B|}$$
showing that the independence dimension of $\phi(x,y)$ is unbounded -- i.e. $\phi(x,y)$ has the independence property.
\end{proof}

\subsection{op-Dimension as an Analog of dp-Rank: ICT- and IRD-patterns}

Thus far, we have seen op-dimension through the lens of the ``stability-like'' analysis of op-ranks and $n$-MOP.  On the other hand, op-dimension can also be characterized using analysis similar to that done on dp-rank; indeed, op-dimension $n$ can be seen as a close analog of dp-rank $n$.  With this in mind, we introduce another alternative definition of op-dimension.  Compare this to the definition of dp-rank given by Definition 2.1 and 2.2 of \cite{GuinHill}.

\begin{thm}\label{Thm_OPDimDpRank}
 Fix a partial type $\pi(x)$ over a parameter set $A$ and $n < \omega$.  The following are equivalent:
 \begin{enumerate}
  \item $\opd(\pi) \le n$;
  \item For all formulas $\varphi(x,y)$, for all indiscernible sequences $\langle b_q : q \in \QQ \rangle$ over $A$, and all $a \models \pi$, there exists $C_0 < ... < C_n$ a convex partition of $\mathbb{Q}$ such that, for each $i \le n$, the set 
  $$\{ q \in C_i : \,\,\models \varphi(a, b_q) \}$$
   is either finite or cofinite in $C_i$.
 \end{enumerate}
\end{thm}

\begin{proof}
 (1) $\Rightarrow$ (2): Suppose that $\opd(\pi) > n$, hence $\pi$ has $(n+1)$-MOP.  Fix an indiscernible picture $( \mathcal{A}_{n+1}, g)$ in $\pi(\mathbb{U})$ and a formula $\varphi(x, y)$ witnessing this.  For each $i \le n$, let $\mathcal{C}_i$ be the set of all $<_i$-cuts of $A$, and let $\mathcal{C} = \prod_{i \le n} \mathcal{C}_i$.  We can multi-order $\mathcal{C}$ via
 \[
  \langle X_0, ..., X_n \rangle \le_i \langle X'_0, ..., X'_n \rangle \text{ iff } X_i \subseteq X'_i.
 \]
 Moreover, for each $c = \langle X_0, ..., X_n \rangle \in \mathcal{C}$, choose $b_c = \langle b_{c,0}, ..., b_{c,n} \rangle \in \mathbb{U}_y^{n+1}$ such that, for all $i \le n$,
 \[
  X_i = \{ a \in A : \models \varphi(g(a), b_{c,i}) \}.
 \]
 Finally, choose a sequence $c_0, c_1, ...$ from $\mathcal{C}$ such that, if $j < k < \omega$, then $c_j <_i c_k$ for all $i \le n$.  Let $f : (n+1) \rightarrow \omega$ be any function.  By $<_i$-density of $A$ for each $i \le n$, there exists $a \in A$ such that
 \[
  \models \varphi(g(a), b_{c_j,i}) \text{ iff } j > f(i)
 \]
 for each $i \le n$.  By compactness and Ramsey's Theorem, there exists $\langle b'_q : q \in \QQ \rangle$ indiscernible and $a' \models \pi$ such that
 \[
  \models \varphi(a', b'_{q,i}) \text{ iff } q > i
 \]
 for each $i \le n$.  Let $\psi(x; y_0, ..., y_n)$ be the formula that holds if evenly many of $\varphi(x, y_i)$ holds for $i \le n$.  Let $\sim$ be the natural convex equivalence relation on $\mathbb{Q}$ generated by $\psi$, namely
 \[
  q \sim r \text{ iff } (\forall q')(q < q' \le r \Rightarrow \models [ \psi(a', b'_q) \leftrightarrow \psi(a', b'_{q'}) ] ).
 \]
 Then, $\sim$ has exactly $n+1$ classes, each infinite.  Thus, we see that $\langle b'_q : q \in \QQ \rangle$, $a$, and $\psi(x, \overline{y})$ is a witness to the failure of (2).
 
 (2) $\Rightarrow$ (1): Suppose (2) fails, witnessed by $\varphi(x, y)$, $\langle b_q : q \in \QQ \rangle$, and $a \models \pi$.  Since $\varphi$ has NIP, we know it has finite alternation rank.  Therefore, by possibly trimming down the sequence and replacing $\varphi$ with $\neg \varphi$, we may assume that $C_0 < ... < C_{n+1}$ is a convex partition of $\mathbb{Q}$, with each $C_i$ infinite, and
 \[
  \models \varphi(a, b_q) \text{ iff } q \in C_i \text{ for some } i \le n+1 \text{ even.}
 \]
 Now we show that $\opr_{n+1}(\pi, \varphi) = \infty$, showing that $\opd(\pi) > n$.
 
 Fix $K < \omega$ and choose $\overline{\sigma} = \langle \sigma_k : k < K \rangle \in ({}^{n+1} 2)^K$.  Suppose we have constructed $q_{i,k} \in (C_i \cup C_{i+1})$ for each $i \le n$ and $k < K$ so that $q_{i,k} \in C_i$ if and only if $\sigma_k(i) = 0$.  Moreover, assume that if $k_0 < k_1 < K$, then
 \begin{enumerate}
  \item $\sigma_{k_0}(i) = \sigma_{k_1}(i) = 0$ implies $q_{i,k_0} < q_{i,k_1}$, and
  \item $\sigma_{k_0}(i) = \sigma_{k_1}(i) = 1$ implies $q_{i,k_1} < q_{i,k_0}$.
 \end{enumerate}
 That is, the sequences $q_{i,k}$ approach the cut between $C_i$ and $C_{i+1}$.  For each $i \le n$, choose $q^*_i \in \mathbb{Q}$ anywhere between all the $q_{i,k}$ for $\sigma_k(i) = 0$ and the $q_{i,k}$ for $\sigma_k(i) = 1$.  Thus, by indiscernibility, 
 \[
  \pi(x) \cup \{ \varphi(x, b_{q_{i,k}})^{i + \sigma_k(i) \ (\mathrm{mod}\ 2)} : k < K, i \le n \} \cup \{ \varphi(x, b_{q^*_i})^{\eta(i)} : i \le n \}
 \]
 is consistent for all $\eta \in {}^{n+1} 2$.  Hence, by induction,
 \[
  \opr_{n+1}(\pi, \varphi) \ge K.
 \]
 Since $K$ was arbitrary, we see that $\opr_{n+1}(\pi, \varphi) = \infty$, as desired.
\end{proof}

We define the notion of an IRD-pattern, given in Definition III.7.1 of \cite{Shelah} and Section 5 of \cite{Adler}, which closely resembles an ICT-pattern (used for dp-rank).  In \cite{Shelah}, Shelah notes that ``IRD'' is an abbreviation for ``independent orders.''  Shelah only considers infinite IRD-patters, but we will diverge from this and consider only finite patterns.

\begin{defn}\label{Defn_oICTPattern}
 Fix a partial type $\pi(x)$, $n < \omega$, and $\alpha$ an ordinal.  Consider a sequence of formulas $\overline{\psi} = \langle \psi_i(x, y_i) : i < n \rangle$ and a sequence $\overline{b} = \langle b_{j,i} : j < \alpha, i < n \rangle$ where each $b_{j,i}$ is of the same sort as $y_i$.  We say that $\langle \overline{\psi}, \overline{b} \rangle$ forms an \emph{IRD-pattern in $\pi(x)$ of depth $n$ and length $\alpha$} if, for all $f : n \rightarrow \alpha$, the following type is consistent
 \[
  \pi(x) \cup \{ \neg \psi_i(x, b_{j,i}) : i < n, j < f(i) \} \cup \{ \psi_i(x, b_{j,i}) : i < n, f(i) \le j < \alpha \}.
 \]
\end{defn}

\begin{lemma}\label{Lem_OPRank000}
 Fix a partial type $\pi(x)$.  Then $\opd(\pi) \le n$ if and only if there exists no IRD-pattern in $\pi$ of depth $n+1$ and length $\omega$.
\end{lemma}

\begin{proof}
 ($\Rightarrow$): Suppose there exists an IRD-pattern in $\pi$ of depth $n+1$ and length $\omega$, say $\langle \overline{\psi}, \overline{b} \rangle$.  Let $\varphi(x; y_0, ..., y_n)$ be the formula that holds if and only if an even number of $\psi_i(x, y_i)$ hold and let $c_j = \langle b_{j,0}, ..., b_{j,n} \rangle$ for each $j < \omega$.  For each strictly monotonic $f : (n+1) \rightarrow \omega$, the type
 \begin{align}\label{Eq_ICToType}
  \pi(x) \cup & \{ \varphi(x, c_j) : f(i-1) \le j < f(i) \text{ for even } i \le n+1 \} \cup \\ \nonumber & \{ \neg \varphi(x, c_j) : f(i-1) \le j < f(i) \text{ for odd } i \le n+1 \},
 \end{align}
 is consistent, where we interpret $f(-1) = 0$ and $f(n+1) = \omega$.  By Ramsey's Theorem and compactness, we may assume that the $c_q$ are indexed by $q \in \QQ$ and that $\overline{c} = \langle c_q : q \in \QQ \rangle$ is indiscernible.  Fix $f : (n+1) \rightarrow \mathbb{Q}$ such that $f(i) = i$ for all $i \le n$ and fix $a$ a realization of \eqref{Eq_ICToType}.  Then $\overline{c}$ and $a$ are witnesses to the fact that $\pi(x)$ has op-dimension $> n$ (as in Theorem \ref{Thm_OPDimDpRank} (2)).
 
 ($\Leftarrow$): Suppose $\pi(x)$ has op-dimension $> n$, witnessed by $\varphi(x, y)$, $\langle b_q : q \in \QQ \rangle$, and $a \models \pi$ (as in Theorem \ref{Thm_OPDimDpRank} (2)).  Since $\varphi$ is NIP, $\varphi$ has finite alternation rank, hence there exists a minimal finite convex partition $\mathcal{C}$ of $\mathbb{Q}$ so that, for each $C \in \mathcal{C}$, there exists $D \subseteq C$ cofinite in $C$ such that, for all $q, r \in D$, $\models \varphi(a, b_q) \leftrightarrow \varphi(a, b_r)$.  Since this is a witness to the op-dimension being greater than $n$, there exists $C_0 < C_1 < ... < C_{n+1}$ from $\mathcal{C}$ with alternating majority truth value of $\varphi(a, b_q)$.  Let $\psi_i(x, y)$ be either $\varphi(x,y)$ or $\neg \varphi(x,y)$ such that, for all $i \le n$ and cofinitely many $q \in C_i$, $\models \psi_i(a, b_q)$ if and only if $i$ is odd.  By indiscernibility over $A$ and compactness, we see that $\langle \psi_i : i \le n \rangle$ together with $\langle b_{i + 1/(j+1)} : j < \omega, i \le n \rangle$ form an IRD-pattern of depth $n+1$ and length $\omega$ in $\pi(x)$.
\end{proof}

We see now that there is an obvious relationship between dp-rank and op-dimension.

\begin{defn}\label{Defn_ICTPattern}
 Fix a partial type $\pi(x)$, $n < \omega$, and $\alpha$ an ordinal.  Consider a sequence of formulas $\overline{\psi} = \langle \psi_i(x, y_i) : i < n \rangle$ and a sequence $\overline{b} = \langle b_{j,i} : j < \alpha, i < n \rangle$ where each $b_{j,i}$ is of the same sort as $y_i$.  We say that $\langle \overline{\psi}, \overline{b} \rangle$ forms an \emph{ICT-pattern in $\pi(x)$ of depth $n$ and length $\alpha$} if, for all $f : n \rightarrow \alpha$, the following type is consistent
 \[
  \pi(x) \cup \{ \neg \psi_i(x, b_{j,i}) : i < n, j < \alpha, j \neq f(i) \} \cup \{ \psi_i(x, b_{f(i),i}) : i < n \}.
 \]
 We say that a type $\pi(x)$ has \emph{dp-rank} $\ge n$ if there exists an ICT-pattern in $\pi$ of depth $n$ and length $\omega$.  We denote this by $\dpr(\pi) \ge n$.
\end{defn}

The next proposition is straightforward, and implicitly shown in \cite{Adler}, but we give a proof here for completeness.

\begin{prop}\label{Prop_dpop}
 Let $\pi(x)$ be a partial type with finite dp-rank.  Then,
 \[
  \opd(\pi) \le \dpr(\pi).
 \]
\end{prop}

\begin{proof}
 Fix $n > \omega$ and let $\overline{\psi} = \langle \psi_i(x, y_i) : i < n \rangle$ together with $\overline{b} = \langle b_{j,i} : j < \omega, i < n \rangle$ be an IRD-pattern of depth $n$ and length $\omega$ in $\pi(x)$.  Let
 \[
  \varphi_i(x; y_{0,i}, y_{1,i}) = \neg [ \psi_i(x, y_{0,i}) \leftrightarrow \psi_i(x, y_{1,i}) ]
 \]
 and let
 \[
  c_{j,i} = \langle b_{2j, i}, b_{2j+1, i} \rangle.
 \]
 Notice that $\langle \varphi_i : i < n \rangle$ together with $\langle c_{j,i} : j < \omega, i < n \rangle$ is an ICT-pattern of depth $n$ and length $\omega$ in $\pi(x)$.  Therefore, $\opd(\pi) \ge n$ implies $\dpr(\pi) \ge n$.
\end{proof}

In particular, if $T$ is dp-minimal (e.g., o-minimal), then $\opd(\UU) \le \dpr(\UU) \le 1$.

Many proofs in the literature establishing the existence of an ICT-pattern implicitly go through an IRD-pattern.  For example, the proof Fact 2.7 of \cite{DGL} first builds an IRD-pattern, then an ICT-pattern from it as in the proof of Proposition \ref{Prop_dpop} above.

In \cite{KOUdpmin}, it is shown that dp-rank is sub-additive in the following sense:
\[
 \dpr(\tp(a,b/C)) \le \dpr(\tp(a/C)) + \dpr(\tp(b/C\cup\{a\})).
\]
This is proved using the technology of mutually indiscernible sequences.  We adapt this for the op-dimension setting using something called almost mutually indiscernible sequences.

\subsection{Almost-indiscernible Sequences}

\begin{defn}\label{Defn_AlmostMutuallyIndisc}
 Fix a set $X$, a collection of sequences
 \[
  \mathcal{J} = \left\{ \langle b_{j,i} : j \in J_i \rangle : i \in X \right\},
 \]
 and a set of formulas
 \[
  \Delta( y_{k,i} )_{k < K_i, i \in X}.
 \]
 We say that $\mathcal{J}$ is $\Delta$-mutually-indiscernible if, for all sequences
 \[
  j_{0,i} < ... < j_{K_i-1,i} \text{ and } \ell_{0,i} < ... < \ell_{K_i-1,i}
 \]
 from $J_i$ for each $i \in X$, and for all $\delta \in \Delta$, we have that
 \[
  \models \delta( b_{j_{k,i},i} )_{k < K_i, i \in X} \leftrightarrow \delta( b_{\ell_{k,i},i} )_{k < K_i, i \in X}.
 \]
 (We note that this depends heavily on the partition of variables in formulas in $\Delta$.) For a set of parameters $A$, we say that $\mathcal{J}$ is \emph{almost mutually indiscernible over $A$} if, for each formula $\delta$ over $A$ as above, there exists $J'_i \subseteq J_i$ finite for each $i \in X$ such that the collection of sequences
 \[
  \left\{ \langle b_{j,i} : j \in (J_i \setminus J'_i) \rangle : i \in X \right\}
 \]
 is $\delta$-mutually-indiscernible.  We say that $\langle b_j : j \in J \rangle$ is \emph{almost indiscernible over $A$} if $\{ \langle b_j : j \in J \rangle \}$ is almost mutually indiscernible over $A$ (where $|X| = 1$).
\end{defn}

\begin{lemma}\label{Lem_Limits}
 Let $\mathcal{J} = \left\{ \langle b_{j,i} : j \in J_i \rangle : i \in X \right\}$ be a set of almost mutually indiscernible sequences, $\delta( y_{k,i} )_{k < K_i, i \in X}$ any formula (over any parameter set), and $\sigma_i : \omega \rightarrow J_i$ a strictly monotone function for each $i \in X$.  Then, there exists $M_i < \omega$ for each $i \in X$ and $t < 2$ such that, for all $M_i < j_{0,t} < ... < j_{K_i-1,i} < \omega$ for each $i \in X$,
 \[
  \models \delta( b_{\sigma_i(j_{k,i}),i} )^t_{k < K_i, i \in X}.
 \]
 That is, there is a ``limit truth value'' for $\delta$ under $\langle \sigma_i : i \in X \rangle$.
\end{lemma}

\begin{proof}
 Write $\delta$ as $\delta(a; y_{k,i} )_{k < K_i, i \in X}$ for $\delta(x; y_{k,i} )_{k < K_i, i \in X}$ a formula over $\emptyset$.  Since $T$ is NIP, $\delta$ is NIP, so suppose it has independence dimension $< N$.
 
 Suppose the conclusion fails.  We build a sequence with alternating truth values on $\delta$ to get a contradiction.  First, choose for each $i \in X$,
 \[
  0 < j^0_{0,i} < ... < j^0_{K_i-1,i} < \omega
 \]
 arbitrarily so that $\models \neg \delta( a; b_{\sigma_i(j^0_{k,i}),i} )_{k < K_i, i \in X}$ (by assumption, this exists).  Now, suppose that $j^\ell_{0,i} < ... < j^\ell_{K_i-1,i}$ is constructed for $\ell \ge 0$, $i \in X$ so that
 \[
  \models \delta( a; b_{\sigma_i(j^\ell_{k,i}),i} )^{\ell \ (\mathrm{mod}\ 2)}_{k < K_i, i \in X}.
 \]
 Let $M_i = j^\ell_{K_i-1,i}$.  By assumption, these $\langle M_i : i \in X \rangle$ and $t = \ell \ (\mathrm{mod}\ 2)$ do not satisfy the conclusion.  Therefore, there exists, for each $i \in X$,
 \[
  M_i < j^{\ell+1}_{0,i} < ... < j^{\ell+1}_{K_i-1,i} < \omega
 \]
 such that
 \[
  \models \delta( a; b_{\sigma_i(j^{\ell+1}_{k,i}),i} )^{\ell+1 \ (\mathrm{mod}\ 2)}_{k < K_i, i \in X}.
 \]
 Notice that, for each $\eta \in {}^N 2$, the formula
 \[
  \theta_\eta( z_0, ..., z_{N-1} ) = \exists x \left( \bigwedge_{w < N} \delta(x; z_w)^{\eta(w)} \right)
 \]
 is over $\emptyset$.  Therefore, by almost mutual indiscernibility of $\mathcal{J}$, we may assume that the sequence
 \[
  \langle \langle b_{\sigma_i(j^\ell_{k,i}),i} : k < K_i, i \in X \rangle : \ell < \omega \rangle
 \]
 is $\{ \theta_\eta : \eta \in {}^N 2 \}$-indiscernible (since $\mathcal{J}$ is merely \emph{almost} mutually indiscernible, we may have to remove a finite portion of the beginning).  Now, for each $\eta \in {}^N 2$, we have by definition
 \[
  \models \bigwedge_{\ell < N} \delta( a; b_{\sigma_i(j^{2\ell+\eta(\ell)}_{k,i}),i} )^{\eta(\ell)}_{k < K_i, i \in X}.
 \]
 By $\{ \theta_\eta : \eta \in {}^N 2 \}$-indiscernibility, we get that
 \[
  \models \exists x \left( \bigwedge_{\ell < N} \delta( x; b_{\sigma_i(j^\ell_{k,i}),i} )^{\eta(\ell)}_{k < K_i, i \in X} \right).
 \]
 Since $\eta$ was arbitrary, this contradicts the fact that $\delta$ has independence dimension $< N$.
\end{proof}

In particular, if $J_i = \omega$ for all $i \in X$ and $\mathcal{J}$ is almost mutually indiscernible over $\emptyset$, then $\mathcal{J}$ is almost mutually indiscernible over any set of parameters.  We use this develop the notion of limit types.

\begin{defn}\label{Defn_LimitType}
 Let $\mathcal{J} = \left\{ \langle b_{j,i} : j \in J_i \rangle : i \in X \right\}$ be an almost mutually indiscernible sequence over a parameter set $A$, $\overline{y} = \langle y_{k,i} : k < K_i, i \in X \rangle$ a tuple of variables, and $\sigma_i : \omega \rightarrow J_i$ a strictly monotone function for each $i \in X$.  Then, for any set of parameters $B$, define the \emph{limit type of $\mathcal{J}$ in the variables $\overline{y}$ under $\overline{\sigma} = \langle \sigma_i : i \in X \rangle$} as follows: For $\delta(\overline{y})$ over $B$,
 \[
  \delta( \overline{y} ) \in \lim_{\overline{\sigma}} (\mathcal{I} / B)(\overline{y})
 \]
 if and only if there exists $M_i < \omega$ for each $i \in X$ such that, for all $M_i < j_{0,i} < ... < j_{K_i-1,i} < \omega$ for each $i \in X$,
 \[
  \models \delta( b_{\sigma_i(j_{k,i}),i} )_{k < K_i, i \in X}.
 \]
\end{defn}

By Lemma \ref{Lem_Limits} above, this is a complete type over $B$ in the variables $\overline{y}$ (that is consistent by compactness).

Fix an ordinal $\alpha < \omega^2$ and let $M < \omega$ be maximal such that $\omega \cdot M \le \alpha$.  For each $m < M$, define the injection $\sigma_m : \omega \rightarrow \alpha$ as follows:
\[
 \sigma_m(i) = (\omega \cdot m) + i.
\]
With this setup, we get the following lemma:

\begin{lemma}\label{Lem_Alternation}
 Fix a finite set $X$ and fix a parameter set $B$.  Suppose that $\mathcal{J} = \{ \langle b_{j,i} : j \in \alpha \rangle : i \in X \}$ is almost mutually indiscernible over $\emptyset$ but not almost mutually indiscernible over $B$.  Then there exists $i_0 \in X$, $m_i < M$ for each $i \in X \setminus \{ i_0 \}$, $m^*_0 < m^*_1 < M$, and $\delta( \overline{y} )$ over $B$ such that
 \begin{enumerate}
  \item $\neg \delta( \overline{y} ) \in \lim_{ \langle \sigma_{m_i} : i \neq i_0 \rangle + \langle \sigma_{m^*_0} \rangle } ( \mathcal{I} / B ) (\overline{y})$ and
  \item $\delta( \overline{y} ) \in \lim_{ \langle \sigma_{m_i} : i \neq i_0 \rangle + \langle \sigma_{m^*_1} \rangle } ( \mathcal{I} / B ) (\overline{y})$.
 \end{enumerate}
\end{lemma}

\begin{proof}
 Let $\delta( \overline{y} )$ over $B$ witness that $\mathcal{J}$ is not almost mutually indiscernible over $B$.  For each choice of $\overline{m} = \langle m_i : i \in X \rangle \in M^X$, consider the sequence of injections $\overline{\sigma}_{\overline{m}} = \langle \sigma_{m_i} : i \in X \rangle$.  By Lemma \ref{Lem_Limits}, there exists $t_{\overline{m}} < 2$ such that
 \[
  \delta^{t_{\overline{m}}}( \overline{y} ) \in \lim_{\overline{\sigma}_{\overline{m}}} ( \mathcal{I} / B ) ( \overline{y} ).
 \]
 If all values of $t_{\overline{m}}$ are equal, then, by removing finitely many elements, $\delta( \overline{y} )$ has a constant value on $\mathcal{I}$.  This contradicts the fact that $\delta$ witnesses that $\mathcal{I}$ is not almost mutually indiscernible over $B$.
 
 Therefore, there must exist $\overline{m}$ and $\overline{m}'$ such that $t_{\overline{m}} \neq t_{\overline{m}'}$.  By switching one coordinate at a time, there exists $i_0 \in X$, $\overline{m}$, and $\overline{m}'$ such that
 \begin{enumerate}
  \item $m_i = m'_i$ for all $i \neq i_0$,
  \item $m_{i_0} < m'_{i_0}$, and
  \item $t_{\overline{m}} \neq t_{\overline{m}'}$.
 \end{enumerate}
 By possibly swapping $\delta$ for $\neg \delta$, we get the desired conclusion.
\end{proof}

We use this lemma in the next section to derive the sub-additivity of $\opd$.

\section{Sub-additivity of $\opd$}

\subsection{Using Almost Mutually Indiscernible Sequences}

In this subsection, we use the machinery of almost mutually indiscernible sequences discussed above to show that op-dimension is sub-additive.  First, we prove a result analogous to Proposition 4.4 of \cite{KOUdpmin}.

\begin{prop}\label{Prop_OPRank002}
 For $\pi$ a partial type over $A$, the following are equivalent
 \begin{enumerate}
  \item $\opd(\pi) \le n$;
  \item For all $a \models \pi$, ordinals $\alpha < \omega^2$, $L < \omega$, and $\{ \langle b_{j,i} : j \in \alpha \rangle : i < L \}$ almost mutually indiscernible over $A$, there exists $I \subseteq L$ with $|I| \ge L-n$ so that $\{ \langle b_{j,i} : j \in \alpha \rangle : i \in I \}$ is almost mutually indiscernible over $A \cup \{ a \}$.
 \end{enumerate}
\end{prop}

\begin{proof}
 (1) $\Rightarrow$ (2): We show the contrapositive, so suppose (2) fails, witnessed by $a \models \pi$ and $\{ \langle b_{j,i} : j \in \alpha \rangle : i < L \}$ almost mutually indiscernible over $A$.  Clearly $L > n$.  Fix $N < \omega$ and $\theta(x) \in \pi(x)$ arbitrary.
 
 Fix $d < L$ and let $X = \{ d, ..., L-1 \}$.  So long as $|X| \ge L-n$ (i.e., $d \le n$), by assumption, $\{ \langle b_{j,i} : j \in \alpha \rangle : i \in X \}$ is not almost mutually indiscernible over $A \cup \{ a \}$.  By Lemma \ref{Lem_Alternation}, there exists a formula $\delta_d( x; y_{k,i} )_{k < K, i \in X}$ over $A$, $\ell < \omega$, $i_0 \in X$, $m_i < \omega$ for each $i \in X \setminus \{ i_0 \}$, and $m^*_0 < m^*_1 < \omega$ such that, for all
 \[
  ((\omega \cdot m_i) + \ell) < j_{0,i} < ... < j_{K-1,i} < (\omega \cdot (m_i + 1)) \text{ for } i \in X \setminus \{ i_0 \},
 \]
 all $t < 2$, and all $((\omega \cdot m^*_t) + \ell) < j_{0,i_0} < ... < j_{K-1,i_0} < (\omega \cdot (m^*_t + 1))$, we have that
 \[
  \models \delta_d( a; b_{j_{k,i},i} )^t_{k < K, i \in X}.
 \]
 Without loss of generality (rearranging the sequences), we may assume $i_0 = d$.  If $d > 0$ and the $\sigma_{d-1}$ have been constructed, then choose $\ell < \omega$ large enough so that no instance of $\sigma_{d-1}(k,i)$ lies in the intervals between $((\omega \cdot m_i) + \ell)$ and $(\omega \cdot (m_i + 1))$.
 
 Define a function $\sigma_d : (2NK \times X) \rightarrow \alpha$.  For each $d < i < L$ and $k < 2NK$, let
 \[
  \sigma_d(k,i) = (\omega \cdot m_i) + \ell + 1 + k.  
 \]
 For $i = d$ and $k < NK$, let
 \[
  \sigma_d(k,d) = (\omega \cdot m^*_0) + \ell + 1 + k.
 \]
 For $i = d$ and $NK \le k < 2NK$, let
 \[
  \sigma_d(k,d) = (\omega \cdot m^*_1) + \ell + 1 + (k - NK).
 \]
 Notice that $\sigma_d(k,i)$ is strictly increasing in the variable $k$.  For $j < 2N$, define
 \[
  c_{d,j} = \langle b_{\sigma_d(k+Kj,i),i} : k < K, i \in X \rangle.
 \]
 
 This construction terminates when $d = n+1$.  We claim that $\delta_d$ together with $\langle c_{d,j} : j < N \rangle$ for $d \le n$ form an IRD-pattern of depth $n+1$ and length $N$ in $\theta$.
 
 By construction, for each $d \le n$, for all $j < 2N$, we have that
 \[
  \models \delta_d(a; c_{d,j}) \text{ iff } j \ge N.
 \] 
 By almost mutual indiscernibility over $A$ (and choosing our $\ell$ above sufficiently large), we get, for each $\eta : (n+1) \rightarrow N$,
 \[
  \models \exists x \left( \theta(x) \wedge \bigwedge_{d \le n, j < N} \delta_d(x; c_{d,j})^{\text{iff } \eta(d) > j} \right).
 \]
 This yields the desired conclusion.
 
 Since $N$ and $\theta$ were arbitrary, by compactness, $\opd(p) \ge n+1$.
 
 (2) $\Rightarrow$ (1): Suppose that $\opd(p) \ge n+1$, witnessed by an IRD-pattern $\overline{\psi} = \langle \psi_i(x, y_i) : i \le n \rangle$ together with $\overline{b} = \langle b_{j,i} : j < \omega, i \le n \rangle$.  Let $\alpha = \omega \cdot 2$ and let $\L'$ be the language $\L$ expanded by constants $b_{j,i}$ for $j < \alpha$ and $i \le n$ and a constant $a$.  Let $\Sigma$ be the $\L'$-theory expanding $T$ which states that
 \begin{itemize}
  \item [(i)] $a \models \pi$,
  \item [(ii)] $\{ \langle b_{j,i} : j < \alpha \rangle : i \le n \}$ is mutually indiscernible over $A$, and
  \item [(iii)] $\models \psi_i(a, b_{j,i})$ if and only if $\omega \le j < \alpha$.
 \end{itemize}
 Any finite subset of $\Sigma$ is realized (using Ramsey's Theorem for (ii)).  Therefore, this is consistent.  Finally, we show that, for all $i \le n$, $\langle b_{j,i} : j \in \alpha \rangle$ is not almost indiscernible over $A \cup \{ a \}$, witnessed by $\psi_i(a, y)$.  By (iii), $\models \psi_i(a, b_{i,j})$ if and only if $\omega \le j < \omega \cdot 2$.  Therefore, for no finite $J_0 \subseteq \alpha$ do we have that $\langle b_{j,i} : j \in (\alpha \setminus J_0) \rangle$ is $\psi_i(a, y)$-indiscernible.
\end{proof}

\begin{thm}[Sub-additivity of op-dimension]\label{Thm_Subadditivity}
 Suppose $a$ and $b$ are tuples and $A$ is a set of parameters.  Then,
 \[
  \opd(a,b/A) \le \opd(a/A) + \opd(b/A \cup \{ a \}).
 \]
\end{thm}

\begin{proof}
 We use Proposition \ref{Prop_OPRank002} in both directions.  First, suppose that
 \[
  \opd(a/A) = n \text{ and } \opd(b/A \cup \{ a \}) = k,
 \]
 let $\alpha < \omega^2$, and let $\mathcal{J} = \{ \langle b_{j,i} : j \in \alpha \rangle : i < L \}$ be almost mutually indiscernible over $A$.  Then, by Proposition \ref{Prop_OPRank002}, there exists $I \subseteq L$ with $|I| = L - n$ so that $\{ \langle b_{j,i} : j \in \alpha \rangle : i \in I \}$ is almost mutually indiscernible over $A \cup \{ a \}$.  Now, by Proposition \ref{Prop_OPRank002} again, there exists $I' \subseteq I$ with $|I'| = L - n - k$ such that $\{ \langle b_{j,i} : j \in \alpha \rangle : i \in I' \}$ is almost mutually indiscernible over $A \cup \{ a, b \}$.  Since $\mathcal{J}$ was arbitrary, by Proposition \ref{Prop_OPRank002}, this implies that $\opd(a,b/A) \le n+k$.
\end{proof}

\subsection{Alternative Proof of Sub-additivity Using MOPs}

In this subsection, we present a sketch of an alternative proof of the fact that $\opd$ is sub-additive.\footnote{Since the result is already proven, it seems unnecessary to subject the reader to another argument in full detail.} This proof uses our $n$-multi-order properties for the analysis of op-dimension, and really amounts to one main compactness argument.

\medskip
\begin{proof}[Second proof (sketch) of Theorem \ref{Thm_Subadditivity}]  Fix $0<n<\omega$, and suppose $\opd(e_0e_1)\geq n$, where $e_0,e_1$ are elements of sorts $v_0,v_1$ in $\UU$, and on the other hand, suppose $\opd(e_0)=k_0<n$. Of course, we must show that $\opd(e_1/e_0)\geq k_1 = n-k_0$. Suppose $g = (g_0,g_1):\A_n\to\UU^{v_0}\times\UU^{v_1}$ is an indiscernible picture of $\A_n$ and $\phi(v_0v_1,u)$ is some formula of $\L$ such that: (1) $g_0(a)g_1(a)\equiv e_0e_1$ for all $a\in A$; and (2) for every $n$-multi-cut $Z = (Z_0,...,Z_{n-1})$, there are $c_0,...,c_{n-1}\in\UU^u$ such that $Z_i = \left\{a\in A: \,\,\models\phi(g_0(a),g_1(a),c_i)\right\}$ for each $i<n$.

For the compactness argument, we introduce a language $\L^+$ that accommodates several indiscernible pictures and many new constant symbols. (For brevity, we assume that $\UU$ has only one sort, and we (rather blithely) work with function symbols whose arities may be of dimension greater than one.)
\begin{itemize}
\item $\L^+$ has four sorts $X_0,X_1,Y$ and $M$ for $\A_{k_0}$, $\A_{k_1}$, $\A_n$, and $\UU$, respectively. (In particular, these sorts have symbols for all of the necessary structure coming from those models.)

\item $\L^+$ has function symbols 
$$\mathbf{g}_i = Y\to M^{v_i},\,\mathbf{f}_i:X_i\to Y,\,\mathbf{h}_i:X_i\to M^{v_i}\,(i<2).$$
For economy, we may combine $\mathbf{g}_0$ and $\mathbf{g}_1$ into a single function symbol $\mathbf{g}:Y\to M^{v_0}\times M^{v_1}$.

\item For each $i<2$, $\L^+$ has constant symbols $c_a^i$ of sort $X_i$ for each $a\in A$.

\item For each $i<2$, each $\B\in\mathrm{age}(\A_{k_i})$, and each $k_i$-multi-cut $Z = (Z_0,...,Z_{k_i-1})$ of $\B$, $\L^+$ has constant symbols $d_0^i(\B,Z)$, ..., $d_{k_i-1}^i(\B,Z)$.

\item $\L^+$ has two additional constant symbols $e^*_0,e_1^*$ of sorts $M^{v_0}$, $M^{v_1}$, respectively. 
\end{itemize}
Now, we will define a set $\Gamma$ of $\L^+$-sentences so that a model of $\Gamma$ contains indiscernible pictures witnessing that $\tp(a_1/a_0)$ has $k_1$-MOP.
\begin{enumerate}
\item For $i<2$, $\Gamma$ asserts $X_i\models\mlo_{k_i}$ (not $\mlo_{k_i}^*$), and for each $\theta(a_0,...,a_{m-1})\in\diag(\A_{k_i})$, $\theta(c_{a_0}^i,...,c_{a_{m-1}}^i)$ is in $\Gamma$. 

\item $\Gamma$ asserts that $(M,Y,\mathbf{g})$ is elementarily equivalent to $(\UU,\A_n,g)$.

\item For each $i<2$, the sentence $\bigvee_{s:k_i\overset{1:1}{\longrightarrow} n}\textsf{Emb}_i(s)$ is in $\Gamma$, where in turn, for each one-to-one map $s:k_i\to n$, $\textsf{Emb}_i(s)$ is the sentence asserting that $\mathbf{f}_i:X_i\to Y$ is an embedding up to identifying the order relations via $<_j\,\mapsto\, <_{s(j)}$.

\item For $i<2$, for each formula $\psi(w_0,...,w_{m-1})$ of $\L$ if $i=0$ (of $\L(e_0^*)$ if $i=1$), and for all 
elements $a_0,...,a_{m-1}$ and $b_0,...,b_{m-1}$ of $A$, if 
$$\qtp^{\A_{k_i}}(a_0,...,a_{m-1}) = \qtp^{\A_{k_i}}(b_0,...,b_{m-1})$$
then the sentence $I^i_{\psi,a,b} = $
$$\psi(\mathbf{f}_i(c_{a_0}^i),...,\mathbf{f}_i(c_{a_0}^i))\bic \psi(\mathbf{f}_i(c_{b_0}^i),...,\mathbf{f}_i(c_{b_0}^i)) $$
is in $\Gamma$.


\item $\psi(e_0^*)\in\Gamma$ for each $\psi(v_0)\in\tp(e_0)$. Further, $\mathbf{h}_0 = \mathbf{g}_0\circ\mathbf{f}_0$ and $\mathbf{h}_1 = \mathbf{g}_1\circ\mathbf{f}_1$.

\item For $\B\in \mathrm{age}(\A_{k_1})$ and $Z = (Z_0,...,Z_{k_1-1})$ a $k_0$-multi-cut in $\B$, the following sentence $\epsilon^1[\B,Z]$ is in $\Gamma$:
$$\bigvee\begin{cases}
\bigwedge\left\{\phi(\mathbf{g}_0(\mathbf{f}_1(c^1_b)), \mathbf{h}_1(c^1_b),d_j^1(\B,Z))^{b\in Z_j}: b\in B, j<k_1\right\}\\
\\
\bigwedge\left\{\phi(e^*_0, \mathbf{h}_1(c^1_b),d_j^1(\B,Z))^{b\in Z_j}: b\in B, j<k_1\right\}.
\end{cases}$$
\end{enumerate}
We observe that if $\M$ is a model of $\Gamma$, and $\A_{k_1}$ is identified with its representation as the set of constant symbols $\left\{c^1_a:a\in A\right\}$, then by the Pigeonhole Principle, one of two things can be true, either one of which demonstrates that $\opd(e_1/e_0)\geq k_1$:
\begin{enumerate}
\item $\mathbf{h}_1^\M:\A_{k_1}\to M^{v_1}$ is an indiscernible picture of $\A_{k_1}$ that with $\phi'(v_1,u) = \phi(e_0^*,v_1,u)$, shows that if $a\in A$, then $$\opd(\mathbf{h}_i^\M(a)/e_0^*)\geq k_1.$$
\item $\mathbf{h}_1^\M:\A_{k_1}\to M^{v_1}$ is an indiscernible picture of $\A_{k_1}$ that with $\phi''(v_1,uv_0) = \phi(v_0,v_1,u)$, shows that if $a\in A$, then $$\opd(\mathbf{h}_i^\M(a)/e_0^*)\geq k_1.$$
\end{enumerate}
In either case, $\opd(e_1/e_0)\geq k_1$ follows because $e_0^*\mathbf{h}_i^\M(a)\equiv e_0e_1$. Thus, it is enough to verify that $\Gamma$ is finitely satisfiable.

\medskip

If $\Gamma$ is \emph{not} satisfiable, then there are $\B_i\in\mathrm{age}(\A_{k_i})$ -- say that $B_i = \left\{b_0^i<_0\cdots<_0b_{N-1}^i\right\}$ -- formulas $\xi(v_0)\in\tp(e_0)$, $\psi_0(w_0,...,w_{N-1})\in\L$, and $\psi_1(w_0,...,w_{N-1})\in\L(e_0^*)$, and a sentence $\sigma$ of $Th(\UU,\A_n,g)$ such that (up to abusing notation in the transfers $a\mapsto c_a^i$) for all one-to-one $s_i:k_i\to n$ and $t_i<2$ ($i<2$),
$$\sigma,\diag(\B_0),\diag(\B_1),\mathbf{h}_0 = \mathbf{g}_0\circ\mathbf{f}_0, \mathbf{h}_1 = \mathbf{g}_1\circ\mathbf{f}_1,$$
$$\textsf{Emb}_0(s_0),\textsf{Emb}_1(s_1), \psi_0(\mathbf{h}_0B_0)^{t_0},  \psi_1(\mathbf{h}_1B_1)^{t_1}$$
implies 
$$\neg\bigwedge\left\{\epsilon^1[\B_1,Z]:\textrm{$Z$ a $k_1$-multi-cut of $\B_1$}\right\}.$$
A few moments' reflection will convince the reader that this contradicts the assumption that for every $n$-multi-cut $W = (W_0,...,W_{n-1})$ of $\A_n$, there are $c_0,...,c_{n-1}\in\UU^u$ such that $W_i = \left\{a\in A: \,\,\models \phi(g_0(a),g_1(a),c_i)\right\}$ for each $i<n$. This completes the proof sketch.
\end{proof}

\bigskip

\section{Connections to o-Minimality}


\subsection{Equivalence of $\opd$, $\dpr$, and o-Minimal Dimension}

The goal of this subsection is to show that op-dimension, dp-rank, and o-minimal dimension coincide in o-minimal theories.  For a definable set $X$, the op-dimension of $X$ is simply the op-dimension of the partial type $x \in X$, and this is denoted $\opd(X)$.  Similarly define the dp-rank.

\begin{thm}\label{Thm_oMinDim}
 If $T$ is o-minimal (where $<$ is dense) and $X$ is a definable set, then the op-dimension of $X$, the dp-rank of $X$, and the o-minimal dimension of $X$ are equal.
\end{thm}

\begin{proof}
 Suppose that $X \subseteq \UU^m$ has o-minimal dimension $\ge n$.  Then, there exists a projection $\pi : \UU^m \rightarrow \UU^n$ so that $\pi(X)$ has non-empty interior.  That is, there exists an open box $B \subseteq \pi(X)$.  Since the ordering $<$ is dense, there exists an embedding $\sigma : \omega^n \rightarrow B$.  This extends to an embedding $\sigma' : \omega^n \rightarrow X$ via $\pi^{-1}$.  Consider, for each $i < n$, the formula $\psi_i(x,y)$ that holds of $\langle a, b \rangle \in ( \UU^m )^2$ if and only if the $i$th coordinate of $\pi(a)$ is less than the $i$th coordinate of $\pi(b)$.  Then, $\psi_i$ together with $\langle \sigma(0, ..., 0, j, 0, ..., 0) : 0 < j < \omega \rangle$ ($j$ in the $i$th coordinate) form an IRD-pattern of depth $n$ in $x \in X$.  Therefore, the op-dimension of $X$ is $\ge n$.  Moreover, by Proposition \ref{Prop_dpop}, $\dpr(X) \ge \opd(X) \ge n$.
 
 Conversely, suppose the o-minimal dimension of $X \subseteq \UU^m$ is $< n$.  By Corollary \ref{cor:dim-monotonicity} (3), we may suppose $X$ is a cell.  Then, there exists a definable injection $f : X \rightarrow \UU^k$ for some $k < n$.  Hence, by Corollary \ref{cor:dim-monotonicity} (1) and (2), $\opd(X) \le \opd(\UU^k)$.  Since we are working in an o-minimal theory, the op-dimension of $\UU^1$ is $\le 1$.  Therefore, by Theorem \ref{Thm_Subadditivity}, $\opd(\UU^k) \le k < n$, hence $\opd(X) < n$.  Moreover, by sub-additivity of dp-rank (Theorem 4.8 of \cite{KOUdpmin}), $\dpr(X) \le \dpr(\UU^k) \le k < n$.
\end{proof}

This result generalizes to any theory expanding dense linear order with a good cell decomposition.  In fact, an interesting question is how does op-dimension relate to cell decomposition?  Can one develop a notion of cell decomposition from the assumption that a theory expanding dense linear order has op-dimension $\le 1$?

\begin{rem}\label{Rem_Distal}
 Notice that dp-rank and op-dimension coincide on any distal theory (see Definition 2.1 of \cite{Simon2}).  To see this, consider the characterization of op-dimension given in Theorem \ref{Thm_OPDimDpRank} together with the characterization of distality given in Lemma 2.7 of \cite{Simon2} (so called external characterization).  From here one can see that global ``point discrepancies'' cannot exist.  Since o-minimal theories are distal, this (along with the fact that dp-rank and o-minimal dimension coincide) gives another proof of Theorem \ref{Thm_oMinDim}.
\end{rem}

\subsection{$d$-Sub-interpretations in o-Minimal Structures}

In this subsection, re-consider the dimension equivalence just presented in the language of interpretations between structures. Unsurprisingly, it turns out that a ``true'' interpretation of some $\B$ in another structure $\M$ is not quite appropriate, and instead we work with a mapping of $B$ to onto a dense subset of a member of $\M^\eq$. With this adjustment, we find that if ``$d$-sub-interpretation'' of $\B$ in the quotient of an $n$-dimensional definable set of $\M$ exists, then $\opd(Th(\B)) \le n$ -- in essence, this is just a restatement of the results of the previous subsection. As a partial converse, however, we manage to show that every countable op-minimal theory $T$ (in a one-sorted language) -- meaning that $\opd(T)\leq 1$ -- is $d$-sub-interpretable (in fact, $d$-sub-definable) in 1-dimension in a pseudo-o-minimal theory.

\begin{defn}
Assume $\M = (M,<,...)$ is o-minimal. For some $0<n<\omega$, let $D\subseteq M^n$ be a definable set with interior (with respect to the product o-minimal topology), and let $E\subseteq D\times D$ be a definable equivalence relation on $D$ with quotient mapping $\pi_E:D\to D/E$. Then we shall always understand $D/E$ to be endowed with the final topology induced by $\pi_E$; that is, $U\subseteq D/E$ is open if and only if 
$\pi_E^{-1}U =\left\{d\in D: \pi_E(d)\in U\right\}$ is open in the subspace topology on $D$.
\end{defn}

We now formulate our weakened notion of interpretability of a structure $\B$ in a topological structure $\M$. 

\begin{defn}
Fix a model $\B$ of $T$, and let $\M = (M,<,...)$ be some o-minimal structure. The data of an {\em $n$-dimensional $d$-sub-interpretation $I$ of $\B$ in $\M$} is the following:
$$I = \left(X,E,\left(\phi^I(v_0,...,v_{k-1})\right)_{\phi(x_0,...x_{k-1})\in \textsf{QF}(\L)},f\right)$$
where $X\subseteq M^r$ ($n\leq r<\omega$) is definable of o-minimal dimension $n$; $E\subseteq X\times X$ is a definable equivalence relation on $X$; for each quantifier-free formula $\phi(x_0,...,x_{k-1})$ of $\L$ (i.e. $\phi\in\textsf{QF}(\L)$), $\phi^I(v_0,...,v_{k-1})$ is a formula of $\L_\M$ such that $|v_i| = r$ for each $i<k$ and $\phi^I(\M)\subseteq X^k$; and $f:B\to X/E$ is a one-to-one mapping. For these data to amount to an $n$-dimensional $d$-sub-interpretation, we require that:
\begin{itemize}
\item For each quantifier-free formula $\phi(x_0,...,x_{k-1})$ of $\L$ (i.e. each $\phi\in\textsf{QF}(\L)$), for all $b_0,...,b_{k-1}\in B$, 
$$\B\models \phi(b)\,\,\iff\,\,\M\models\exists v_0 ... v_{k-1}\left(\bigwedge_{i<k}v_i{\in} f(b_i)\wedge \phi^I(v_0,...,v_{k-1})\right).$$
\item  $f[A]$ is dense in $X/E$.
\end{itemize}
Naturally enough, we will say that $T$ is {\em $n$-dimensionally o-minimally $d$-sub-interpretable} if there are $\B\models T$, $\M$ an o-minimal structure, and an $n$-dimensional $d$-sub-interpretation of $\B$ in $\M$. When the equivalence relation $E$ is trivial (i.e. $E=1_X$), then we say ``sub-definable'' instead of ``sub-interpretable.''

We remark that there is nothing exceedingly special about o-minimality in this definition (or the previous one). Indeed, largely the same formulations would work for weakly o-minimal, pseudo-o-minimal, or (it seems) any theory with a definable topology.
\end{defn}

\begin{fact}
Let $\B\models T$, and let $I = (X,E,(\phi^I),f)$ be an $n$-dimensional $d$-sub-interpretation of $\B$ in an o-minimal structure $\M$. If $f[A] =X/E$, then $I$ is an interpretation of $\A$ in $\M$ in the classical sense.
\end{fact}

\begin{thm}
Assume $T$ eliminates quantifiers (in a language with a single sort). If $T$ is $n$-dimensionally o-minimally $d$-sub-interpretable, then $\opd(T) \leq n$ -- meaning that $\opd(\{x{=}x\})\leq n$ where $x$.
\end{thm}
\begin{proof}
Let $\B\models T$, and let $I = (X,E,(\phi^I),f)$ be a $d$-sub-interpretation of $\B$ in an o-minimal structure $\M = (M,<,...)$. Absorbing the the parameters of the formulas $X,E, \phi^I$ into the language, we assume that $I$ is over $\emptyset$. Also, assuming $\dim(X)\leq n-1$, we show that $T$ cannot have $n$-MOP. For a contradiction, suppose (as we may, by QE) $\psi(x,y)\in \L$ is a quantifier-free formula and $g:A\to B$ is an indiscernible picture of $\A_n$ in $\B$ such that for every multi-cut $(X_0,...,X_{n-1})$, there are $b_i$ ($i<n$) such that $X_i = \left\{a:\B\models\psi(g(a),b_i)\right\}$. The following claim is a relatively straightforward consequence (by compactness) of the fact that $f[B]$ is dense in $X/E$.

\begin{claim}
There are elementary extensions $\A_n\preceq\A'$, $\B\preceq\B'$, $\M\preceq\M'$ and functions $g':A;\to B'$, and $f':B\to X(\M')/E(\M')$ such that:
\begin{enumerate}
\item $f'\subseteq f$, and $I = (X(\M'), E(\M'),(\phi^I),f')$ is a $d$-sub-interpretation of $\B'$ in $\M'$.
\item $g\subseteq g'$, and $g'$ is an indiscernible picture of $\A'$ in $\B'$ pattered on $\EM: \aa\mapsto g\aa$, and for every multi-cut $(X_0,...,X_{n-1})$, there are $b_0,...,b_{n-1}\in B'$ such that $X_i = \left\{a:\B' \models\psi(g(a),b_i)\right\}$ for each $i<n$.
\item Relative to the 0-definable structure on $X(\M')/E(\M')$, the composition $(f'\circ g')\r A$ is an indiscernible picture of $\A_n$ in $\M'$.
\end{enumerate}
\end{claim}

Before sketching a demonstration the claim, we first complete the proof the theorem from it. (For clarity, we will abuse notation now by suggesting that $x,y,w$ are really single variables rather than tuples; this a fiction due abbreviating.) As $\psi$ is quantifier-free, 
$$\B'\models \psi(b,b')\,\,\iff\,\,\M'\models\exists v,v'\left(v{\in} f'(b)\wedge v'{\in} f'(b')\wedge \psi^I(v,v')\right).$$
whenever $b,b'\in B$ are of the appropriate sorts. In particular, if $(X_0,...,X_{n-1})$ is a multi-cut in $\A_n$, then choosing $b_0,...,b_{n-1}\in B'$ appropriately, we have
\begin{align*}
X_i &= \left\{a\in A: \B'\models\psi(g'(a),b_i)\right\}\\
&=\left\{a\in A: \M'\models \exists vv'\left(v{\in} f'(g'(a))\wedge v'{\in} f'(b_i)\wedge \psi^I(v,v')\right) \right\}
\end{align*}
Thus, the indiscernible picture $(f'\circ g')\r A$ and the formula implicit above show that 
$$n\leq \opd(X)\leq \dim(X(\M')) = n-1$$
which is impossible in light of Theorem \ref{Thm_oMinDim}.


\begin{proof}[Proof (sketch) of claim]
We will work in a language with three sorts $B$, $M$, $A$ on which the symbols of $\L$, $\L_\M$, and those of $\mlo_n$, respectively, are imposed; between these sorts, we will also have function symbols $\mathbf{f}:A\to B$ and $\mathbf{g}:B\to X\subseteq M^r$. We include constants for all elements of the countable model $\A_n$. Finally, to account for defined multi-cuts, we include function symbols $\mathbf{h}_0,...,\mathbf{h}_{n-1}:A^n\to B$. Now, the truth of the claim boils down to verifying that the following set of sentences $\Gamma$ of this language is finitely-satisfiable.
\begin{itemize}
\item $\Gamma$ says $B$ is a model of $T$, $M$ is a model of $Th(\M)$, and $A$ is a model of $\mlo_n^*$ with the countable model as a substructure via the added constants.
\item For each $k<\omega$ and each quantifier-free-complete $k$-type $q(x)$ of the language of $\mlo_n$, for each formula $\phi(x)\in \tp^\B(ga)$ where $a\in q(\A_n)$, 
$$(\forall x_0...x_{k-1}\in A)\left(q(x)\cond \phi(\mathbf{g}x)\right)$$
is in $\Gamma$.

\item For each $k<\omega$ and each quantifier-free-complete $k$-type $q(x)$ of the language of $\mlo_n$, for each formula $\phi(x)$ of $\L_\M$ such that $Th(\M)$ implies $\phi\cond X^k$, for all $k$-tuples $a,b$ over the set of constants naming the countable model $\A_n$
$$q(a)\wedge q(b)\cond (\phi(\mathbf{fg}a)\bic \phi(\mathbf{fg}b))$$
is in $\Gamma$.

\item The sentence,
$$(\forall z_0...z_{n-1}\in A)(\forall x\in A)
\bigwedge_{i<n} (x<_iz_i\bic \psi(\mathbf{g}(x),\mathbf{h}_i(z)))
$$
is in $\Gamma$.

\item $\Gamma$ asserts that $\mathbf{f}$ is the mapping associated with a $d$-sub-interpretation using $X$, $E$ and $(\phi^I)_\phi$ of the $B$ in $M$.
\begin{itemize}
\item For each quantifier-free formula $\phi(x_0,...,x_{k-1})$ of $\L$, the sentence
$$\forall x\left(\phi(x)\bic\exists v_0...v_{k-1}\bigwedge_{i<k}E(v_i,\mathbf{f}x_i)\wedge \phi^I(v)\right)$$
\item Density: Suppose $X_0,...,X_{N-1}$ are the cells of $X$, each with a definable bijection $e_i:X_i\to R_i$ onto a definable rectangle $R_i\subseteq M^{d_i}$ with $d_i\leq n$ ($i<N$, $d_i\leq n$) 
$$\bigwedge_{i<N}(\forall x,y\in M^{d_i})\left[ \Pi_{\ell<N}(x_\ell,y_\ell)\subseteq R_i\cond (\exists z\in B)e(\mathbf{f}(z))\in\Pi_{\ell<N}(x_\ell,y_\ell) \right]$$
\end{itemize}

\end{itemize}
\end{proof}
This completes the proof of the theorem.
\end{proof}

An immediate consequence of the previous theorem (and Morley-ization) is the following, giving a loose characterization of any theory interpretable in an o-minimal theory as ``stable in a sufficiently loose sense.''

\begin{cor}
Let $\M$ be an o-minimal structure. For any structure $\A$, if $\A$ is interpretable in a model of $Th(\M)$, then for some $n<\omega$, $Th(\A)$ does not have $n$-MOP.
\end{cor}

Recall that a structure is pseudo-o-minimal just in case it is elementarily equivalent to an ultraproduct of o-minimal structures, and a theory is pseudo-o-minimal just in case it has a pseudo-o-minimal model. 

\begin{prop}
Let $T$ be any op-minimal theory in a countable one-sorted language $\L$ -- i.e., $\opd(T)\leq 1$. Then $T$ is 1-dimensionally $d$-sub-definable in an pseudo-o-minimal structure. 
\end{prop}
\begin{proof}
We fix a countable model $\B_0$ of $T$. We define a language $\L^+$ with two sorts $X,Y$ and a function symbol $\mathbf{f}:X\to Y$; further, $X$ carries the whole signature of $\L$, and for each formula $\phi(x_0,...,x_{n-1})$ of $\L$ (where the $x_i$s are single variables), let $R_\phi$ be an $n$-ary relation symbol on $Y$. Finally, let $0,1$ be constant symbols on $Y$, and let $<$ be a binary relation symbol on $Y$. Let $\L^{++}$ be the further expansion of $\L^+$ to included $\QQ$ as a set of constant symbols on $Y$ (with $0,1$ playing themselves). Let $\D$ be the set of all pairs $(q_0,q_1)\in\QQ^2$ such that $0\leq q_0<q_1\leq 1$. For each $F\subset_\fin e\diag(\B_0)$ and each $D\subset_\fin \D$, let $\Sigma_{F,D}$ be the following set of sentences of $\L^{++}$: 
\begin{itemize}
\item $\Sigma_{F,D}$ says $(Y,<)\models$ DLO, and $(q_0<q_1)\in\Sigma_{F,D}$ for all $(q_0,q_1)\in D$.
\item $\Sigma_{F,D}\models$ ``$T\cup F$ on $X$.''
\item $(\forall x\in X)( 0\leq \mathbf{f}(x)\leq 1)$ and ``$\mathbf{f}$ is one-to-one'' are in $\Sigma_{F,D}$.
\item $\left\{R_\theta(\mathbf{f}(b_0),...,\mathbf{f}(b_{n-1})):\theta(b_0,...,b_{n-1})\in F\right\}\subseteq \Sigma_{F,D}$.
\item For each $(q_0,q_1)\in D$, $(\exists x\in X)(q_0<\mathbf{f}(x)<q_1)$ is in $\Sigma_{F,D}$.
\end{itemize}
Let $\L^{++}[F,D]$ be the sub-language of $\L^{++}$ whose signature on the $Y$-part is restricted to 
$$\{<\}\cup\QQ\cup\{R_\phi:\textrm{$\phi$ is a sub-formula of a member of $F$}\}.$$
A subset $X\subseteq\L^{++}[F,D]$ is called a fragment of $\L^{++}[F,D]$ if it is closed under boolean combinations, changes of variables and taking sub-formulas.

\begin{obs}\label{obvious-obs}
Suppose that for any $F\subset_\fin e\diag(\B_0)$ and $D\subset_\fin \D$, there is a model $(\B_{F,D},\M_{F,D},f_{F,D})$ of $\Sigma_{F,D}$ in which $\M_{F,D}$ is a pseudo-o-minimal structure. Then $T$ is 1-dimensionally pseudo-o-minimally $d$-sub-definable. 
\end{obs}
\begin{proof}[Proof of Observation]
Let $\Psi$ be any non-principal ultrafilter on the set $\P_\fin(e\diag(\B))\times\P_\fin(\D)$ (where for any set $X$, $\P_\fin(X)$ is the set of finite subsets of $X$), and let $(\C,\M,f)=\Pi_{F,D}(\B_{F,D},\M_{F,D},f_{F,D})/\Psi$. Clearly, $\C = \Pi_{F,D}\B_{F,D}/\Psi$ is a model of $T$, and $\M = \Pi_{F,D}\M_{F,D}/\Psi$ is a pseudo-o-minimal expansion of $(M,<^\M,0^\M,1^\M)$ . Moreover, 
$$\left([0,1],1_{[0,1]}, \left(R_\phi\right)_{\phi\in\textsf{QF}(\L)},f\right)$$
is a $d$-sub-definition of $\B$ in $\M$.
\end{proof}


\begin{claim}
Let $F\subset_\fin e\diag(\B_0)\setminus T$, $\phi\in e\diag(\B_0)\setminus T$, and $D\subset_\fin\D$. Then $\Sigma_{F,D}$ has model $(\B,\M,f)$ such that $\M$ is pseudo-o-minimal.
\end{claim}
\begin{proof}[Proof of claim]
For a finite set $S\subset_\fin \emph{Sent}(\L^{++}[F,D])$, let us say that $S$ is $F,D$-good if there is model $(\B,\M,f)\models\Sigma_{F,D}\cup S$ such that for any $\phi(x,y)\in\L^{++}[F,D]$ with ($x$ a single variable, $x,y$ both of sort $Y$, and every sub-formula of $\phi$ in the language of $\M$), if $\phi(x,y)$ is in the fragment of $\L^{++}[F,D]$ generated by $\Sigma_{F,D}\cup S$, then $\phi$ is \emph{not} an obstruction to o-minimality of $\M$ (meaning that there is a number $k<\omega$ such that for any $c\in M^y$, $\phi(\M,c)$ is indeed equal to the union of $\leq k$ open intervals and $\leq k$ points). 

Enumerating $\emph{Sent}(\L^{++}[F,D])$ as $\{\phi_j\}_{j<\omega}$, we define a tree $W\subseteq 2^{<\omega}$ consisting of those $\sigma\in 2^{<\omega}$ such that $\left\{\phi_j^{\sigma(j)}:j<|\sigma|\right\}$ is $F,D$-good.
As $W$ is a finitely-branching tree, if it is infinite, then by K\"onig's Lemma, we recover an infinite branch $f:\omega\to 2$ of $W$, and this $f$ encodes a complete pseudo-o-minimal theory $T_{F,D}$ such that 
$$\Sigma_{F,D}\subseteq T_{F,D}\subseteq  \emph{Sent}(\L^{++}[F,D]).$$
For a contradiction, then, we assume that $W$ is finite -- in particular, $2^{<\omega}\setminus W$ has a finite set $\{\sigma_0,...,\sigma_{N-1}\}$ of minimal elements. 

Let $b = (b_0,...,b_{m-1})$ and $q = (q_0,...,q_{m'-1})$ enumerate all of the elements of $\B_0$ and $\QQ$, respectively, that appear in any $\phi_j\in \sigma_i^{-1}(1)$, $i<N$.  Without loss of generality, we may also assume that $|\sigma_i|=|\sigma_0|=\ell$ for each $i<N$.  Moreover, by the Robinson Joint Consistency Theorem, for each $i<N$, the conjunction $\bigwedge_{j<\ell}\phi_j^{\sigma_i(j)}$ is equivalent modulo $\Sigma_{F,D}$ to a conjunction $\bigwedge_{s<t_i}\psi_{i,s}$, where each $\psi_{i,s}$ is of the form 
$$\eta_{i,s}(\mathbf{f}b,q)\wedge \theta_{i,s}(b)\cond \theta'_{i,s}(\mathbf{f}b)$$
and where $\eta_{i,s}$ is a quantifier-free-complete type in the language of order, $\theta_{i,s}$ is an $\L$-formula (the language of $\B_0$), and $\theta'_{i,s}$ is a formula of $\L^{++}[F,D]$ that has no $X$-sorted variables or constants at all and does not involve $<$. 

Now, for any proper extension $\sigma_i\subset\tau\in 2^{<\omega}$ and any $(\B,\M,f)\models \Sigma_{F,D}\cup\left\{\phi_j^{\tau(j)}\right\}_{j<|\tau|}$, there are $i<N$, $s<t_i$, and $k<n$ such that the partitioned formula,
$$\theta_{i,s}'(y_k;y_0...y_{k-1}y_{k+1}...y_{n-1})$$
is an obstruction to o-minimality in $\M$. For all $i<N$, $s<t_i$, $k<n$, this formula ``pulls back'' to an $\L$-formula, $\xi_{i,s}^k(v;w_1...w_{n-1})$. (Let $\zeta_{i,s}^k(x;y_1...y_{n-1})$ be the un-pulled-back formula.) We make an easily, if tediously, verifiable observation:

\begin{obs}
Suppose there is a finite model $(\C,<_0^\C,<_1^\C)\models\mlo_2$ such that for any one-to-one map $g:C\to B_0$ and for each $i_0,s_0,k_0$,  there is multi-cut $(X_0,X_1)$ such that  for any $i_1,s_1,k_1$, one cannot choose $b_0,b_1\in B^{n-1}$, $\B_0\preceq\B$, so that $X_j = \left\{c: \xi_{i_j,s_j}^{k_j}(g(c),b_j)\right\}$ for both $j=0,1$. Then there are numbers $e(i,s,k)$ such that
$$\Sigma_{F,D}\cup\left\{\forall \yy\textrm{``$\zeta_{i,s}^k(x,\yy)$ is the union of $\leq e(i,s,k)$ points and open intervals''}\right\}_{i,s,k}.$$
is consistent.
\end{obs}

Now, since $\opd(T)\leq 1$, there must be such a finite model $(\C,<_0^\C,<_1^\C)$ of $\mlo_2$; otherwise, we would have $\opr_2(\{x{=}x\},\{\xi^k_{i,s}(v,\ww)\}_{i,s,k})=\infty$. Consequently, we have a contradiction to the definition of $\sigma_0,...,\sigma_{N-1}$, so $W$ must be infinite -- which proves completes the proof of the claim and of the proposition.
\end{proof}
\end{proof}

\section*{Acknowledgements}

We thank Pierre Simon for pointing out the fact that op-dimension and dp-rank coincide in distal theories.


\end{document}